%% file: uh-v2.tex
\documentclass[reqno,12pt]{amsart}
\usepackage{amsmath,amssymb,epsfig,color,here,wasysym,stmaryrd,diagbox,tikz-cd,url}

\numberwithin{equation}{section}

\newtheorem{theorem}{Theorem}[section]

\newtheorem{lemma}[theorem]{Lemma}
\newtheorem{proposition}[theorem]{Proposition}
\newtheorem{definition}[theorem]{Definition}
\newtheorem{corollary}[theorem]{Corollary}

\theoremstyle{remark}

\newtheorem{remark}[theorem]{Remark}

\newcommand{\vanish}[1]{}

\begin{document}

\title{On reduced unicellular hypermonopoles} 

\author[Robert Cori and G\'abor Hetyei]{Robert Cori \and G\'abor Hetyei}

\address{Labri, Universit\'e Bordeaux 1, 33405 Talence Cedex, France.
\hfill\break
WWW: \tt http://www.labri.fr/perso/cori/.}

\address{Department of Mathematics and Statistics,
  UNC Charlotte, Charlotte NC 28223-0001.
WWW: \tt http://webpages.charlotte.edu/ghetyei/.}

\date{\today}


\begin{abstract}
The problem of counting unicellular hypermonopoles by the number of their
hyperedges is equivalent to
describing the cycle length distribution of a product of two circular
permutations, first solved by Zagier. The solution of this problem has
also been used in the study of the cycle graph model of Bafna and
Pevzner and of related models in mathematical biology. In this paper we
develop a method to compute the finite number of reduced unicellular
hypermonopoles of a given genus. The problem of 
representing any hypermap as a drawing is known to be simplifiable to
solving the same problem for reduced unicellular hypermonopoles. We also
outline a correspondence between our hypermap model, the cycle graph model of
Bafna and Pevzner, and the polygon gluing model of Alexeev and
Zograf. Reduced unicellular hypermonopoles correspond to reduced
objects in the other models as well, and the notion of genus is the
same.  
\end{abstract}

\maketitle

\section*{Introduction}

In the study of the combinatorics of the  symmetric group many authors
have been interested in the statistics of cycle lengths of products
of pairs of permutations. In this note we revisit the particular case of
counting the cycles of the product of two circular permutations, or
dually finding the number of decompositions of a given permutation as a
product of two circular permutations.   
  
Our renewed interest in this topic comes from the study of {\em hypermaps}.
In a recent paper~\cite{Cori-Hetyei-hypertrees} we showed that the
problem of drawing a hypermap may be reduced to considering the same
problem for a {\em unicellular hypermonopole} of the same genus. A hypermap
$H=(\sigma,\alpha)$ is a pair of permutations generating a transitive
permutation group. A hypermonopole  is a hypermap with a single
vertex, that is, $\sigma$ is a circular permutation, it is called {\em
  unicellular} if it has only one face, meaning  that $\alpha^{-1}
\sigma$  is  also a  circular permutation.   
The main result of this short paper concerns the enumeration of
unicellular hypermonopoles without {\em buds}, meaning  that $\alpha$ has no
fixed point,  we call such a unicellular hypermonopole {\em reduced} . 

Our  main tool is the enumeration formula obtained by
Zagier~\cite{Zagier} for the number of unicellular  hypermonopoles
having a given number of cycles. Notice that the number of cycles $k$
and the genus $g$ of a unicellular hypermonopole  of $S_n$ satisfy $k=
n-2g$. A combinatorial bijective proof of Zagier's formula was given
by Cori, Marcus and Schaeffer~\cite{Cori-Marcus-Schaeffer}.

In  the field of genome rearranging,  Bafna and
Pevzner~\cite{Bafna-Pevzner} introduced a cycle graph model that is
cryptomorphic to studying the product of a pair of circular
permutations. Their aim  was  to determine the minimum number of
``transpositions''\footnote{The transpositions in biology are
different from those considered in algebra, they may be obtained by a
composition of two usual transpositions} needed to reduce a permutation
(considered as a word) to the identity permutation. This allows to
model the evolution of the DNA of some  viruses.
These questions have motivated several researchers to focus on the
products of circular permutations. This is the case of the work of
A.\ Hultman in his thesis~\cite{Hultman}, and more recently  that of
Alexeev and Zograf~\cite{Alexeev-Zograf} who introduced gluings  of
polygons to represent these products.  

Our paper is organized as follows. In the Preliminaries
we recall the notions of the theory of hypermaps and state Zagier's
formula and some of its reformulations. Section~\ref{sec:basic} is devoted to
the description of the relationship between Zagier's factorization problem, 
counting unicellular hypermonopoles, and the models of Bafna and
Pevzner~\cite{Bafna-Pevzner} and of Alexeev and
Zograf~\cite{Alexeev-Zograf}. In Section~\ref{sec:counting} we give
our main result expressing the number of reduced unicellular
hypermonopoles on $n$ points with a given number of hyperedges. We
already observed in~\cite{Cori-Hetyei-hypertrees} that the number of all
reduced unicellular hypermonopoles of a given genus is finite. Our main
result allows to count these finite numbers explicitly, and
we give the first values of these numbers.

It is worth noting that reduced unicellular hypermonopoles correspond to
cycle graphs having breakpoints everywhere in the model of Bafna and
Pevzner~\cite{Bafna-Pevzner} and that the genus of the corresponding
polygon gluing diagram in the work of  Alexeev and
Zograf~\cite{Alexeev-Zograf} is the same as the genus of the
corresponding unicellular hypermonopole. The three models are intimately
related, hence the counting problem we solved has also some significance
in the related models in mathematical biology.

\section{Preliminaries}
\label{sec:prelim}

\subsection{Hypermaps and their two disk diagrams}

A {\em hypermap} $(\sigma,\alpha)$ is a pair of permutations of a set
$\{1,2,\ldots,n\}$ generating a transitive permutation group. It is used
to represent a (connected) hypergraph on an oriented
surface. The {\em points} 
$1,2,\ldots, n$ are the points of incidence between vertices and
hyperedges. The cycles of $\sigma$ list these points around the
vertices in counterclockwise order, whereas the cycles of $\alpha$ list
these points hyperedges in clockwise order. The cycles of $\alpha^{-1}\sigma$
represent then the faces of the hypermap $(\sigma,\alpha)$. The {\em
  genus $g(\sigma,\alpha)$} of the surface on which such a hypermap may
be drawn is given by the following formula due to Jacques~\cite{Jacques}: 
\begin{equation} 
\label{eq:genusdef}
n + 2 -2g(\sigma,\alpha) = z(\sigma) + z(\alpha) + z(\alpha^{-1}
\sigma),
\end{equation}
where $z(\pi)$ denotes the number of cycles of the permutation $\pi$.
This paper is motivated by the following observation 
in~\cite{Cori-Hetyei-hypertrees}: using a sequence of {\em topological
  hyperdeletions $(\sigma,\alpha)\mapsto (\sigma,\alpha\delta)$}  and of
{\em topological hypercontractions $(\sigma,\alpha)\mapsto
  (\gamma\sigma,\gamma\alpha)$}, each hypermap may be reduced to a
  hypermap $(\sigma',\alpha')$ of the same genus, such that
  $z(\sigma')=1$, that is, $(\sigma',\alpha')$ is a {\em hypermonopole},
  and $z(\alpha^{-1}\sigma')=1$, that is, $(\sigma',\alpha')$ {\em
    unicellular}. We refer the interested reader for the detailed
  description of the process to~\cite{Cori-Hetyei-hypertrees}. The final
  conclusion is that if we are able to draw the unicellular
  hypermonopole $(\sigma',\alpha')$ by some means in the plane then we
  may easily extend such a figure to a drawing of $(\sigma,\alpha)$ by
  replacing the vertex $\sigma'$ with a noncrossing partition having
  $z(\sigma)$ parts and the face $\alpha^{-1}(\sigma)$ with a
  noncrossing partition having $z(\alpha^{-1}\sigma)$ parts. Noncrossing
  partitions, introduced by Kreweras~\cite{Kreweras}, are partitions of
  the set $\{1,2,\ldots,n\}$ such that the polygons representing the
  parts do not cross if we place the points in the cyclic order of
  $(1,2,\ldots,n)$ on a circle.  

There are infinitely many unicellular hypermonopoles of a fixed genus:
trivially, for each $n$ the hypermap $(\sigma,\alpha)$ with
$\sigma=(1,2,\ldots,n)$ and $\alpha=(1)(2)\cdots (n)$ is a unicellular
hypermonopole of genus zero. Fortunately, in the case of unicellular
hypermonopoles it is easy to remove or
reinsert {\em buds}: a bud is a fixed point $i$ of $\alpha$. For $n\geq
2$, the removal of $i$ from the cycles representing $\sigma$ and
$\alpha$ results in a hypermap $(\sigma',\alpha')$, which is still a
hypermonopole (as $z(\sigma')=z(\sigma)=1$) and it is also still
unicellular: the action of $\alpha'^{-1}\sigma'$ on a
$j\not\in\{i,\sigma^{-1}(i)\}$ is the same as that of
$\alpha^{-1}\sigma$, whereas 
$\alpha'^{-1}\sigma'$ takes $\sigma^{-1}(i)$ into    
$\alpha'^{-1}(\sigma(i))$. (Note that neither $\sigma^{-1}(i)$ nor
$\sigma(i)$ is equal to $i$, as $\sigma$ is circular and has at least
$2$ points.) Hence the cycle representing $\alpha'^{-1}\sigma'$ is
obtained from the cycle representing $\alpha^{-1}\sigma$ by simply
removing the point $i$ from the cyclic list. Clearly if we are able to
draw $(\sigma',\alpha')$ in the plane, adding a bud to he figure amounts
to adding a single point.
\begin{definition}
We call a unicellular hypermonopole {\em reduced} if it contains no bud.
\end{definition}  
It has been first observed in~\cite{Cori-Hetyei-hypertrees} that there
are only finitely many reduced unicellular hypermonopoles of a fixed
genus.
\begin{lemma}
\label{lem:genusbound}
If $(\sigma,\alpha)$ is genus $g$ a reduced unicellular hypermonopole on $n$
points then $2g+1\leq n\leq 4g $ holds. 
\end{lemma}  
\begin{proof}
Substituting $z(\sigma)=1$ and $z(\alpha^{-1}\sigma)=1$ into (\ref{eq:genusdef})
we obtain 
\begin{equation}
\label{eq:genusdefr}  
n = z(\alpha)+2g(\sigma,\alpha). 
\end{equation}
Since each cycle of $\alpha$ has at least $2$ elements we get
$z(\alpha)\leq n/2$, yielding the upper bound for $n$. The lower bound
is a direct consequence of $z(\alpha)\geq 1$. 
\end{proof}

\subsection{Products of two circular permutations}

A permutation is {\em circular} if it has exactly one cycle. In his
paper we will need the number of pairs of circular permutations of
$\{1,2,\ldots,n\}$ whose product has exactly $k$ cycles. The answer to
this question was first given by Zagier~\cite[application 3 of Theorem
  1]{Zagier}.
\begin{theorem}[Zagier]
The probability that the product of two cyclic permutations of
$\{1,2,\ldots,n\}$ has $k$ cycles is
$$
P(n,k)=\frac{1+(-1)^{n-k}}{(n+1)!} c(n+1,k).
$$
Here $c(n+1,k)=|s(n+1,k)|$ is the number of permutations of $\{1,2,\ldots,n+1\}$
with $k$ cycles, and $s(n+1,k)$ is a Stirling number of the first kind. 
\end{theorem}
The first purely combinatorial proof of this result was provided by
Cori, Marcus and Schaeffer~\cite[Corollary
  1]{Cori-Marcus-Schaeffer}. 
Note that $P(n,k)=0$ if $n-k$ is odd. This is obvious: all circular
permutations have the same parity, hence the product of two circular
permutations must be an even permutation. The parity of a
permutation is the parity of the number of its cycles of even length in
its cycle decomposition. This number can not be even if $n-k$ is
odd. After noting that one of the two circular permutations may be fixed
to be $(1,2,\ldots,n)$, Zagier's result may be restated in combinatorial
terms as follows.  
\begin{theorem}
\label{thm:circperms}
The number of circular permutations $\psi$ of $\{1,\ldots,n\}$ such that
the product $(1,\ldots,n)\psi$ has exactly $k$ cycles is
\begin{align}
  \label{eq:Hn}
  H(n,k)&=
  \begin{cases}
    c(n+1,k)/\binom{n+1}{2} & \mbox{if $n-k$ is even,}\\
    0 & \mbox{if $n-k$ is odd}. 
  \end{cases}  
\end{align}  
\end{theorem}
\begin{table}[h]
\begin{tabular}{l||c|c|c|c|c|c|c|c|c|c|}
\diagbox{$n$}{$k$}& $1$&$ 2$ & $ 3$ & $ 4$ & $ 5$ & $ 6$ & $ 7$ & $ 8$ & $ 9$ & $ 10 $\\
\hline
0& 1& & & & & & & & & \\
1& 0& 1& & & & & & & & \\
2& 1& 0& 1& & & & & & & \\
3& 0& 5& 0& 1& & & & & & \\
4& 8 &0 & 15 & 0 & 1& & & & & \\
5& 0 & 84 &  0 & 35 &  0 & 1  & & & & \\
6& 180& 0& 469& 0& 70& 0&  & & &  \\
7& 0& 3044& 0& 1869& 0& 126& 0& 1  & & \\
8& 8064& 0& 26060& 0& 5985& 0& 210& 0& 1 & \\
9& 0& 193248& 0& 152900& 0& 16401& 0& 330& 0& 1\\
\end{tabular}
\caption{The values of $H(n,k)$ for $n\leq 9$.}
\label{tab:Hnk}
\end{table}

Table~\ref{tab:Hnk} shows the values of $H(n,k)$ for $n\leq 9$. It is
worth noting that when  
$n-k$ is even, the sign of $s(n+1,k)$ is negative, hence we may also
replace $c(n+1,k)$ with $-s(n+1,k)$ in~(\ref{eq:Hn}) above. 
The numbers $H(n,k)$ were later rediscovered by
A.\ Hultman in his MS Thesis~\cite{Hultman} who
defined them in terms of counting alternating cycles in the cycle graph
of a permutation. The numbers $H(n+1,k)$ were named {\em Hultman
  numbers} in the work of Doignon and Labarre~\cite{Doignon-Labarre},
and they are listed as sequence A164652 in the
Online Encyclopedia of Integer Sequences~\cite{OEIS}.   
The equivalence of the two definitions is made
apparent in~\cite[Corollary 1]{Bona-Flynn}, which is based on a result
of Doignon and Labarre~\cite{Doignon-Labarre}.  Citing
Stanley~\cite{Stanley-EC2}, M.\ B\'ona and
R.\ Flynn~\cite[p.\ 931]{Bona-Flynn} also published (\ref{eq:Hn}) for these
numbers. A simple proof of~(\ref{eq:Hn}) (relying on Hultman's
definition) was also  found by S.\ Grusea and A.\ Labarre~\cite[Section
  7]{Grusea-Labarre}.  We will use the following lemma of S.\ Grusea and
A.\ Labarre~\cite[Lemma 8.1]{Grusea-Labarre}. 
\begin{lemma}[Grusea-Labarre]
\label{lem:GL}
The numbers $H(n,k)$ satisfy
$$  
\sum_{k=0}^n H(n,k) x^k=\frac{(x)^{(n+1)}-(x)_{n+1}}{(n+1)n}
$$
Here $(x)^{(n+1)}=x\cdot(x+1)\cdots (x+n)$ and
$(x)_{n+1}=x\cdot(x-1)\cdots (x-n)$ are falling, respectively rising
factorials (Pochhammer symbols). 
\end{lemma}

\section{Basic facts about unicellular hypermonopoles}
\label{sec:basic}

As a direct consequence of the definitions we may observe the following.

\begin{remark}
\label{rem:2cycles}
For a unicellular hypermonopole $(\sigma,\alpha)$ the permutations $\sigma$
and $\pi=\alpha^{-1}\sigma$, representing the unique vertex,
respectively, the unique face of the hypermap, are circular
permutations. Conversely, given a pair $(\sigma,\pi)$ of circular
permutations of the same set, there is a unique unicellular
hypermonopole $(\sigma,\alpha)$ satisfying $\pi=\alpha^{-1}\sigma$.   
\end{remark}
Indeed, the unique $\alpha$ satisfying $\pi=\alpha^{-1}\sigma$ is
\begin{align}
\label{eq:hf}
\alpha&=\sigma\pi^{-1}.
\end{align}  
Regardless of $\alpha$, the pair $(\sigma,\alpha)$ is always a hypermap
because $\sigma$ is a circular permutation, and any permutation group
containing is transitive. Theorem~\ref{thm:circperms} and
Remark~\ref{rem:2cycles} have the following consequence. 
\begin{corollary}
\label{cor:Hn}
The number of unicellular hypermonopoles $(\sigma,\alpha)$ satisfying
$\sigma=(1,\ldots,n)$ and $z(\alpha)=k$ is the number
$H(n,k)$ given in~(\ref{eq:Hn}). 
\end{corollary}  

In the rest of this section we show that unicellular hypermonopoles are
bijectively equivalent to two other combinatorial models for the
Hultman numbers $H(n+1,k)$. 

The first model we consider is the {\em cycle graph model} introduced by
Bafna and Pevzner~\cite{Bafna-Pevzner}. We present it using the same
simplification that was introduced in~\cite{Doignon-Labarre} where the
vertex $n+1$ is identified with $0$, and we adjust that model even
further. Let us fix the circular
permutation $\sigma=(0,1,\ldots,n)$ and let $\pi$ be any circular permutation of
the set $\{0,1,\ldots,n\}$. The {\em cycle graph} $G(\pi)$ of the permutation
$\pi$ is a digraph on the vertex set $\{0,1,\ldots,n\}$ whose edges are
colored with two colors: 
\begin{enumerate}
\item the black edges go from $i$ to $\pi^{-1}(i)$ (modulo $n+1$) for
  $0\leq i\leq n$;
\item the grey edges go from $i$ to $i+1$ (modulo $n+1$) for
  $0\leq i\leq n$.    
\end{enumerate}  
Each vertex is the head, respectively tail of one edge of each color,
hence the cycle graph may be uniquely decomposed into disjoint
color-alternating cycles. Note that even though edges do not repeat in
such cycles, vertices may occur twice. To remedy this slight confusion,
we introduce two copy of each vertex $i$: a negative copy $i^-$ and a
positive copy $i^+$. Each negative vertex $i^-$ will be
the head of a black edge whose tail is $\pi(i)^+$ and it will be the
tail of a grey edge whose head is $(i+1)^+$. Equivalently,
each positive vertex $i^+$ will be  the head of a grey edge whose tail
is $(i-1)^-$ and it will be the tail of a black edge whose head is
$\pi^{-1}(i)^+$. Instead of using colors we will label the black edges
with $\pi^{-1}$ and the grey edges with $\sigma$. For example, for $n=7$
and $\pi=(0,4,1,6,2,5,7,3)$ we obtain the following two cycles
$$
0^- \xrightarrow{\sigma} 1^+ \xrightarrow{\pi^{-1}} 4^-
\xrightarrow{\sigma} 5^+ \xrightarrow{\pi^{-1}} 2^- \xrightarrow{\sigma}
3^+ \xrightarrow{\pi^{-1}} 7^- \xrightarrow{\sigma} 0^+
\xrightarrow{\pi^{-1}} 3^- \xrightarrow{\sigma} 4^+
\xrightarrow{\pi^{-1}} 0^- \quad\mbox{and} 
$$
$$
6^-\xrightarrow{\sigma} 7^+ \xrightarrow{\pi^{-1}} 5^-
\xrightarrow{\sigma} 6^+ \xrightarrow{\pi^{-1}} 1^-
\xrightarrow{\sigma} 2^+ \xrightarrow{\pi^{-1}} 6^-
$$
Using this notation one may notice immediately that these cycles may be
uniquely reconstructed from the positive vertices only: we may identify
the first cycle with $(1,5,3,0,4)$ and the second cycle with
$(7,6,2)$. Observe next that
$$\alpha=(1,5,3,0,4)(7,6,2)=\sigma\pi^{-1},$$
that is, $\pi=\alpha^{-1}\sigma$. The cycle graph of $\pi$ may be
identified with the unicellular hypermonopole $(\sigma,\alpha)$ whose
only vertex is $\sigma$ and only face is $\pi=\alpha^{-1}\sigma$. The
number of alternating cycles in $G(\pi)$ is $z(\alpha)$, the number of
hyperedges. 
 
Bafna and Pevzner~\cite{Bafna-Pevzner} call a pair $(i,\pi(i)$ a {\em
  breakpoint} if $\pi(i)\neq \sigma(i)$. In our setting $(i,\pi(i))$ is a
breakpoint if and only if $\alpha^{-1}\sigma(i)\neq \sigma(i)$,
equivalently $\sigma(i)$ is not a bud of $(\sigma,\alpha)$. Reduced
unicellular hypermonopoles bijectively correspond to circular
permutations for which every pair $(i,\pi(i))$ is a breakpoint, see
Lemma~\ref{lem:breakpoint} in the next section.  

The other model is the one introduced by Alexeev and
Zograf~\cite{Alexeev-Zograf}. Consider a $2n$ sided polygon whose
boundary consists of $n$ black sides followed by $n$ grey sides, the
black sides are oriented in the the counterclockwise direction and the
grey sides are oriented in the clockwise direction, as shown in
Fig.~\ref{fig:gluing}.    
\begin{figure}[h]
  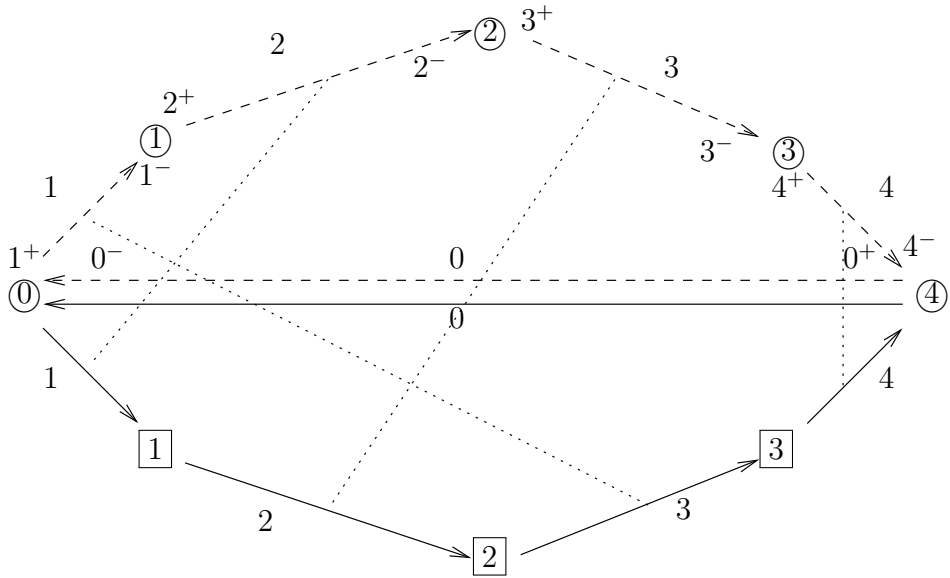
  \caption{A polygon gluing diagram for $\pi=(0,2,3,1,4)$}
  \label{fig:gluing}
\end{figure}  
Pairwise gluing of black sides with gray sides (respecting orientation)
gives an orientable topological surface without boundary of topological
genus $g\geq 0$ (the genus g depends on the gluing). We cut the polygon
along the diagonal connecting the vertex $n$ and the vertex $0$, and we
add a directed edge of each color from $n$ to $0$. By assuming that
these added edges will be glued together we don't change the genus.  
We number all edges by their tail end, and we use the gluing pattern to
define the circular permutation $\pi=(\pi_0,\pi_1,\ldots,\pi_n)$: we
define $\pi_i$ as the grey edge that is glued with the black edge
$i$. As in the previous model, we define $\sigma$ as the circular
permutation $(0,1,\ldots,n)$. For the gluing pattern shown in
Fig.~\ref{fig:gluing} we obtain 
$\pi=(0,2,3,1,4)$.  As before, let us define
$\alpha=\sigma\pi^{-1}$, in our example we obtain
$\alpha=(0)(1,4,2)(3)$. According to Alexeev and
Zograf~\cite{Alexeev-Zograf} it is easy to see that the alternating
cycles of $G(\pi)$ are in bijection with the vertices of the glued
polygon. We can make it easier to see this by adding the signed labels
$i^-$ and $i^+$ along each grey edge labeled $i$ as shown in
Fig.~\ref{fig:gluing}. For example, the alternating cycle
$$
0^- \xrightarrow{\sigma} 1^+ \xrightarrow{\pi^{-1}} 3^-
\xrightarrow{\sigma} 4^+ \xrightarrow{\pi^{-1}} 1^- \xrightarrow{\sigma}
2^+ \xrightarrow{\pi^{-1}} 0^-$$
corresponds to the identification $\textcircled{0}=
\fbox{2}=\textcircled{3}=\fbox{3}=\textcircled{1}=\textcircled{0}$.
The verification of the details is left to the reader.

\section{Counting reduced unicellular hypermonopoles}
\label{sec:counting}

In this section we express the number of reduced unicellular hypermonopoles in
terms of the number $H(n,k)$.
In doing so, the following lemma will be
useful. 
\begin{lemma}
\label{lem:breakpoint}  
A unicellular hypermonopole $(\sigma,\alpha)$ satisfying $\sigma=(1,\ldots,n)$
is reduced if and only if there is no $i\in\{1,\ldots,n\}$ that the
circular permutation  
$\pi=\alpha^{-1}\sigma$ takes into $i+1$. Here addition is performed
modulo $n$. 
\end{lemma}  
Indeed, $\alpha^{-1}\sigma(i)=i+1$ is equivalent to $\alpha(i+1)=i+1$. 
Let us also note the following consequence of the proof of
Lemma~\ref{lem:genusbound}. 
\begin{corollary}
\label{cor:kbound}  
If $(\sigma,\alpha)$ is a reduced unicellular hypermonopole on $n$
points then $1\leq z(\alpha)\leq n/2$ holds. 
\end{corollary}  

\begin{proposition}
\label{prop:rnk}
  Given $n\geq 2$ and $1\leq k\leq n/2$, the number of
  reduced unicellular hypermonopoles 
  $(\sigma,\alpha)$ satisfying $\sigma=(1,\ldots,n)$ and $z(\alpha)=k$
  is given by
\begin{align}
\label{eq:rnk}
r(n,k)&=\sum_{i=0}^{k-1} (-1)^{i} \binom{n}{i} H(n-i,k-i). 
\end{align}  
\end{proposition}  
\begin{proof}
We compute $r(n,k)$ using inclusion-exclusion. Let  ${\mathcal H}_{n,k}$
be the set of all unicellular hypermonopoles $(\sigma,\alpha)$ satisfying
$\sigma=(1,\ldots,n)$ and 
$z(\alpha)=k$. For each $j\in\{1,\ldots,n\}$, let
${\mathcal H}_{n,k,j}$ be the subset of ${\mathcal H}_{n,k}$ also satisfying
$\alpha^{-1}\sigma(j)=j+1$. Clearly we have
$$
r(n,k)=\left| {\mathcal H}_{n,k}-\bigcup_{j=1}^n {\mathcal H}_{n,k,j}\right|.
$$
Using the inclusion-exclusion formula we obtain
\begin{align}
\label{eq:r1}  
r(n,k)&=\sum_{i=0}^n (-1)^i \sum_{\{j_1,\ldots,j_i\}\subseteq \{1,\ldots,n\}}
\left| {\mathcal H}_{n,k,j_1} \cap \ldots \cap {\mathcal H}_{n,k,j_i}\right|.
\end{align}
First we show that it suffices to perform the summation on the right
hand side only up to $i=k-1$. All hypermaps $(\sigma,\alpha)\in
{\mathcal H}_{n,k,j_1} \cap \ldots  
\cap {\mathcal H}_{n,k,j_i}$ have the property that $j_1,\ldots,j_i$ are fixed
points of $\alpha$. Since $z(\alpha)=k$, we may restrict the summation
in~(\ref{eq:r1}) to $i\leq k$. Furthermore the case $i=k$ is possible
only if the cycles $(j_1),\ldots,(j_k)$ are all the cycles of $\alpha$ in
which case $k=n$ in contradiction with $k\leq n/2$.

From now on let us fix a subset $\{j_1,\ldots,j_i\}$ of
$\{1,\ldots,n\}$ for some $i\leq k-1$. Let $\sigma'$ be the circular
permutation of $\{1,\ldots,n\}-\{j_1,\ldots,j_i\}$ obtained from $(1,\ldots,n)$
by removing the elements $j_1,\ldots,j_i$.  Given any unicellular
hypermonopole $(\sigma,\alpha)\in {\mathcal H}_{n,k,j_1} \cap \ldots \cap
{\mathcal H}_{n,k,j_i}$, let us define the unicellular hypermonopole 
$(\sigma',\alpha')$ on the set of points
$\{1,\ldots,n\}-\{j_1,\ldots,j_i\}$ by the following procedure: 
\begin{enumerate}
\item We define $\pi=\alpha^{-1}\sigma$ as the unique face of
  $(\sigma,\alpha)$.  
\item We define the unique face $\pi'$ of $(\sigma',\alpha')$ as the
  circular permutation obtained by removing the elements
  $j_1,\ldots,j_i$ from from $\pi$ . 
\item The permutation $\alpha'$ is given by $\alpha'=\sigma'\pi'^{-1}$.  
\end{enumerate}
The operation $(\sigma,\alpha)\mapsto (\sigma',\alpha')$ associates to
each element  of ${\mathcal H}_{n,k,j_1} \cap \ldots \cap {\mathcal H}_{n,k,j_i}$ a
unicellular hypermonopole $(\sigma',\alpha')$. The operation is
invertible: to obtain $\pi$ from $\pi'$ we must insert each $j\in
\{j_1,\ldots,j_i\}$ right before $j+1$. Hence we obtain a bijection
between the hypermaps in ${\mathcal H}_{n,k,j_1} \cap \ldots \cap {\mathcal H}_{n,k,j_i}$ and
the set of all unicellular hypermonopoles with vertex
$\sigma'$. Therefore we have
$$
\left| {\mathcal H}_{n,k,j_1} \cap \ldots \cap {\mathcal
  H}_{n,k,j_i}\right|=H(n-i,k-i), 
$$
and the statement is a direct consequence of~(\ref{eq:r1}).
\end{proof}
\begin{corollary}
If $n-k$ is odd then $r(n,k)=0$. As a consequence the least value of $n$
for which $r(n,k)>0$ holds for some $k\leq \frac{n}{2}$ is $n=3$.
\end{corollary}
Indeed, if $n-k$ is odd then all terms $H(n-i,k-i)=0$ appearing on the
right hand side of \eqref{eq:rnk} are zero. 

The values of $r(n,k)$ for $3\leq n\leq 12$ are shown in
Table~\ref{tab:rnk}. They are also listed as sequence A371665 in the
Online Encyclopedia of Integer Sequences~\cite{OEIS}. 
\begin{table}[h]
\begin{tabular}{l||c|c|c|c|c|c|}
\diagbox{$n$}{$k$}& $1$&$ 2$ & $ 3$ & $ 4$ & $ 5$ & $ 6$\\
\hline
3& 1& & & & & \\
4& 0& 1& & & & \\
5& 8& 0& & & & \\
6& 0& 36& 0& & & \\
7& 180& 0& 49& & & \\
8& 0& 1604& 0& 21 && \\
9& 8064& 0& 5144& 0 & & \\
10& 0& 112608& 0& 7680& 0 &\\
11& 604800& 0& 604428& 0& 5445 & \\
12& 0& 11799360& 0& 1669052& 0& 1485\\
\end{tabular}
\caption{The values of $r(n,k)$ for $3\leq n\leq 12$ and $1\leq k\leq
\lfloor n/2\rfloor$.}
\label{tab:rnk}
\end{table}

Combining Lemma~\ref{lem:GL} and Equation~(\ref{eq:rnk}) we obtain the
following formula.

\begin{theorem}
  The numbers $r(n,k)$ satisfy
  $$
  \sum_{k=0}^{\lfloor n/2\rfloor} r(n,k)\cdot x^k =
  \sum_{i=0}^{n-1} \binom{n}{i} (-x)^{i} \cdot
  \frac{(x)^{n-i}-(x)_{n-i}}{(n-i)(n-i+1)}.  
  $$
\end{theorem}  
\begin{proof}
  By Lemma~\ref{lem:GL} the number $H(n-i,k-i)$ is given by
  $$
  H(n-i,k-i)=[x^{k-i}] \frac{(x)^{n-i}-(x)_{n-i}}{(n-i)(n-i+1)}=
[x^k]  x^i\frac{(x)^{n-i}-(x)_{n-i}}{(n-i)(n-i+1)}. 
  $$
The statement now follows from Equation~(\ref{eq:rnk}) after noticing
that we may extend the upper limit of the summation to $n$: 
$$
r(n,k)=\sum_{i=0}^{n} (-1)^{i} \binom{n}{i} H(n-i,k-i)
$$
also holds if we set $H(n,k)=0$ for $k\leq 0$. Lemma~\ref{lem:GL} is
still applicable: the expression $((x)^{(n+1)}-(x)_{n+1})/((n+1)n)$ is a
polynomial of $x$ with zero constant term, containing no negative powers
of $x$.  
\end{proof}
  
Lemma~\ref{lem:genusbound} and Proposition~\ref{prop:rnk} allow us to
compute the number of all reduced unicellular hypermonopoles of a fixed
genus, using the following result. 

\begin{proposition}
The number $u(g)$ of all reduced unicellular hypermonopoles of genus
$g$ is given by
$$
u(g)=\sum_{n=2g+1}^{4g} r(n,n-2g).
$$
\end{proposition}  
\begin{proof}
By~\eqref{eq:genusdefr} a unicellular hypermonopole with $k$ cycles has
genus $g$ if and only if $k=n-2g$ holds. As seen in
Lemma~\ref{lem:genusbound}, $n$ must be at least $2g+1$ and at most
$4g$.  
\end{proof}  

The first $10$ entries of the sequence $\{u(g)\}_{g=1}^{\infty}$ are the
following:

\begin{align*}
  &2, 114, 21538, 8698450, 6113735682, 6641411533106,\\
  &10323616703610338, 21755183272319116818,\\
&59718914489141881419202,
207083242485963591169089778. 
\end{align*}

\section*{Acknowledgments}
The second author wishes to express his heartfelt thanks to Labri, Universit\'e
Bordeaux I, for hosting him as a visiting researcher in Summer 2023 and
in Spring 2024, where a great part of this research was performed. 
This work was partially supported by a
grant from the Simons Foundation (\#514648 to G\'abor Hetyei).

\end{document}

%% file: gluing.pdf_t
\begin{picture}(0,0)%
\includegraphics{gluing.pdf}%
\end{picture}%
\setlength{\unitlength}{3947sp}%
\begin{picture}(5730,3591)(4561,-6655)
\put(9826,-4711){\makebox(0,0)[lb]{\smash{\fontsize{12}{14.4}\normalfont {\color[rgb]{0,0,0}$0^+$}%
}}}
\put(4576,-4936){\makebox(0,0)[lb]{\smash{\fontsize{12}{14.4}\normalfont {\color[rgb]{0,0,0}$\textcircled{0}$}%
}}}
\put(7501,-6586){\makebox(0,0)[lb]{\smash{\fontsize{12}{14.4}\normalfont {\color[rgb]{0,0,0}$\fbox{2}$}%
}}}
\put(5401,-3961){\makebox(0,0)[lb]{\smash{\fontsize{12}{14.4}\normalfont {\color[rgb]{0,0,0}$\textcircled{1}$}%
}}}
\put(5401,-5911){\makebox(0,0)[lb]{\smash{\fontsize{12}{14.4}\normalfont {\color[rgb]{0,0,0}$\fbox{1}$}%
}}}
\put(7501,-3286){\makebox(0,0)[lb]{\smash{\fontsize{12}{14.4}\normalfont {\color[rgb]{0,0,0}$\textcircled{2}$}%
}}}
\put(9376,-4036){\makebox(0,0)[lb]{\smash{\fontsize{12}{14.4}\normalfont {\color[rgb]{0,0,0}$\textcircled{3}$}%
}}}
\put(9301,-5911){\makebox(0,0)[lb]{\smash{\fontsize{12}{14.4}\normalfont {\color[rgb]{0,0,0}$\fbox{3}$}%
}}}
\put(10276,-4936){\makebox(0,0)[lb]{\smash{\fontsize{12}{14.4}\normalfont {\color[rgb]{0,0,0}$\textcircled{4}$}%
}}}
\put(7351,-4711){\makebox(0,0)[lb]{\smash{\fontsize{12}{14.4}\normalfont {\color[rgb]{0,0,0}$0$}%
}}}
\put(4801,-4261){\makebox(0,0)[lb]{\smash{\fontsize{12}{14.4}\normalfont {\color[rgb]{0,0,0}$1$}%
}}}
\put(6226,-3361){\makebox(0,0)[lb]{\smash{\fontsize{12}{14.4}\normalfont {\color[rgb]{0,0,0}$2$}%
}}}
\put(8701,-3511){\makebox(0,0)[lb]{\smash{\fontsize{12}{14.4}\normalfont {\color[rgb]{0,0,0}$3$}%
}}}
\put(10051,-4261){\makebox(0,0)[lb]{\smash{\fontsize{12}{14.4}\normalfont {\color[rgb]{0,0,0}$4$}%
}}}
\put(4801,-5461){\makebox(0,0)[lb]{\smash{\fontsize{12}{14.4}\normalfont {\color[rgb]{0,0,0}$1$}%
}}}
\put(6151,-6361){\makebox(0,0)[lb]{\smash{\fontsize{12}{14.4}\normalfont {\color[rgb]{0,0,0}$2$}%
}}}
\put(8776,-6286){\makebox(0,0)[lb]{\smash{\fontsize{12}{14.4}\normalfont {\color[rgb]{0,0,0}$3$}%
}}}
\put(10051,-5461){\makebox(0,0)[lb]{\smash{\fontsize{12}{14.4}\normalfont {\color[rgb]{0,0,0}$4$}%
}}}
\put(7351,-5086){\makebox(0,0)[lb]{\smash{\fontsize{12}{14.4}\normalfont {\color[rgb]{0,0,0}$0$}%
}}}
\put(5551,-3736){\makebox(0,0)[lb]{\smash{\fontsize{12}{14.4}\normalfont {\color[rgb]{0,0,0}$2^+$}%
}}}
\put(5401,-4186){\makebox(0,0)[lb]{\smash{\fontsize{12}{14.4}\normalfont {\color[rgb]{0,0,0}$1^-$}%
}}}
\put(7126,-3511){\makebox(0,0)[lb]{\smash{\fontsize{12}{14.4}\normalfont {\color[rgb]{0,0,0}$2^-$}%
}}}
\put(7801,-3211){\makebox(0,0)[lb]{\smash{\fontsize{12}{14.4}\normalfont {\color[rgb]{0,0,0}$3^+$}%
}}}
\put(4576,-4711){\makebox(0,0)[lb]{\smash{\fontsize{12}{14.4}\normalfont {\color[rgb]{0,0,0}$1^+$}%
}}}
\put(8926,-4036){\makebox(0,0)[lb]{\smash{\fontsize{12}{14.4}\normalfont {\color[rgb]{0,0,0}$3^-$}%
}}}
\put(10201,-4636){\makebox(0,0)[lb]{\smash{\fontsize{12}{14.4}\normalfont {\color[rgb]{0,0,0}$4^-$}%
}}}
\put(9376,-4261){\makebox(0,0)[lb]{\smash{\fontsize{12}{14.4}\normalfont {\color[rgb]{0,0,0}$4^+$}%
}}}
\put(5101,-4711){\makebox(0,0)[lb]{\smash{\fontsize{12}{14.4}\normalfont {\color[rgb]{0,0,0}$0^-$}%
}}}
\end{picture}%